\documentclass{amsart}

\usepackage{amsmath,amscd}
\usepackage{amssymb}
\usepackage{mathtools}
\usepackage{stmaryrd}

\usepackage[all]{xy}

\usepackage[T1]{fontenc}

\usepackage{times}

\usepackage[T1]{fontenc}
\usepackage{graphicx}
\usepackage[svgnames]{xcolor}
\usepackage{framed}

\usepackage{tikz}

\usepackage{enumitem,kantlipsum}

\usepackage{microtype}

\usepackage{etoolbox}
\apptocmd{\thebibliography}{\raggedright}{}{}

\newcommand*\openquote{\makebox(25,-10){\scalebox{3}{``}}}
\newcommand*\closequote{\makebox(25,-10){\scalebox{3}{''}}}
\colorlet{shadecolor}{White}
\makeatletter
\newif\if@right
\def\shadequote{\@righttrue\shadequote@i}
\def\shadequote@i{\begin{snugshade}\begin{quote}\openquote}
\def\endshadequote{%
  \if@right\hfill\fi\closequote\end{quote}\end{snugshade}}
\@namedef{shadequote*}{\@rightfalse\shadequote@i}
\@namedef{endshadequote*}{\endshadequote}
\makeatother

\author{Liran Shaul}
\address{Department of Algebra, Faculty of Mathematics and Physics, Charles University in Prague, Sokolovsk\'a 83, 186 75 Praha, Czech Republic}

\email{shaul@karlin.mff.cuni.cz}
\dedicatory{Dedicated to Amnon Yekutieli on the occasion of his 60th birthday}

\newtheorem{thm}[equation]{Theorem}
\newtheorem{cthm}{Theorem}

\newtheorem{cor}[equation]{Corollary}
\newtheorem{prop}[equation]{Proposition}

\theoremstyle{definition}

\newtheorem{rem}[equation]{Remark}

\newcommand{\opn}{\operatorname}
\newcommand{\cat}[1]{\operatorname{\mathsf{#1}}}

\newcommand{\mfrak}[1]{\mathfrak{#1}}

\newcommand{\mrm}[1]{\mathrm{#1}}
\newcommand{\mbb}[1]{\mathbb{#1}}

\newcommand{\K}{\mbb{K} \hspace{0.05em}}

\newcommand{\m}{\mfrak{m}}

\newcommand{\p}{\mfrak{p}}
\newcommand{\q}{\mfrak{q}}
\newcommand{\injdim}{\operatorname{inj\,dim}}

\newcommand{\amp}{\operatorname{amp}}

\begin{document}

\title{Open loci results for commutative DG-rings}

\begin{abstract}
Given a commutative noetherian non-positive DG-ring with bounded cohomology which has a dualizing DG-module,
we study its regular, Gorenstein and Cohen-Macaulay loci.
We give a sufficient condition for the regular locus to be open,
and show that the Gorenstein locus is always open. 
However, both of these loci are often empty:
we show that no matter how nice $\mrm{H}^0(A)$ is, 
there are examples where the Gorenstein locus of $A$ is empty.
We then show that the Cohen-Macaulay locus of a commutative noetherian DG-ring
with bounded cohomology which has a dualizing DG-module always contains a dense open set.
Our results imply that under mild hypothesis,
eventually coconnective locally noetherian derived schemes are generically Cohen-Macaulay,
but even in very nice cases, they need not be generically Gorenstein.
\end{abstract}


\setcounter{tocdepth}{1}
\setcounter{section}{-1}

\maketitle

\section{Introduction}

The aim of this paper is to answer the question: how does an eventually coconnective derived scheme look like generically?
In classical algebraic geometry, it is well known that any algebraic variety is generically regular,
so that a typical point on it is the spectrum of a regular local ring.
For more complicated schemes, this need not be the case,
and the regular locus of a noetherian scheme could be empty.
For this reason, the questions of what are the Gorenstein and Cohen-Macaulay loci of a noetherian scheme
are important questions, first considered by Grothendieck. The next quote by To\"{e}n is taken from \cite{To}:

\begin{shadequote}
Derived algebraic geometry is an extension of algebraic geometry whose main purpose is
to propose a setting to treat geometrically special situations (typically bad intersections,
quotients by bad actions,... ), as opposed to generic situations (transversal intersections,
quotients by free and proper actions,... ).
\end{shadequote}

In view of this, it should not be surprising that eventually coconnective derived schemes often behave badly from a homological point of view,
and are often singular everywhere, and even nowhere Gorenstein. 
The main result of this paper shows, however, that even in these special geometric situations where derived schemes are needed,
under very mild hypothesis, eventually coconnective derived schemes are generically Cohen-Macaulay.
Since questions about loci are local, 
we may assume that our derived schemes are affine,
so they are represented by simplicial commutative rings.
Normalizing them using the monoidal Dold-Kan correspondence, one obtains commutative non-positive (in cohomological grading)
differential graded rings. 
These have the additional property that they have divided powers structure.
As this structure will not be needed for the results of this paper,
we will work more generally with arbitrary commutative non-positive DG-rings.
The eventually coconnective assumption is equivalent to the fact that these commutative DG-rings have bounded cohomology.
We impose this assumption, as only in this case the Gorenstein and Cohen-Macaulay properties are understood.
The main result of this paper states:

\begin{cthm}
\leavevmode
\begin{enumerate}[wide, labelwidth=!, labelindent=0pt]
\item Let $A$ be a commutative noetherian DG-ring with bounded cohomology which has a dualizing DG-module.
Then the Gorenstein locus of $A$ is an open subset of $\opn{Spec}(\mrm{H}^0(A))$.
\item For any commutative noetherian ring $B$, there is a commutative noetherian DG-ring $A$ with bounded cohomology such that $\mrm{H}^0(A) = B$, and such that the Gorenstein locus of $A$ is empty.
It holds that $A$ has a dualizing DG-module if and only if $B$ has a dualizing complex.
\item Let $A$ be a commutative noetherian DG-ring with bounded cohomology which has a dualizing DG-module.
Then the Cohen-Macaulay locus of $A$ is a subset of $\opn{Spec}(\mrm{H}^0(A))$ which contains a dense open subset.
\end{enumerate}
\end{cthm}

\section{Preliminaries}

As mentioned above, we work with commutative non-positive differential graded rings.
A complete reference for DG-rings is the book \cite{YeBook}.
Commutative DG-rings, by definition, are graded rings $A = \oplus_{n=-\infty}^0 A^n$,
with a differential of degree $+1$ which satisfies a graded Leibniz rule with respect to multiplication.
The commutativity assumption says that $b\cdot a = (-1)^{i\cdot j}\cdot a \cdot b$ for $a \in A^i$ and $b \in A^j$,
and moreover $a\cdot a = 0$ if $a \in A^i$ with $i$ being odd. 
It follows that $\mrm{H}^0(A)$ is a commutative ring.
Informally, one should think of the geometric space $A$ as being $\opn{Spec}(\mrm{H}^0(A))$,
with some additional homotopical structure. 
In particular, 
the various loci studied in this paper will be subsets of $\opn{Spec}(\mrm{H}^0(A))$ with respect to the Zariski topology.
We say that $A$ has bounded cohomology if there exists 
$m \in \mathbb{Z}$ such that for all 
$n<m$ one has that $\mrm{H}^n(A) = 0$.
In the geometric setting, the terminology for bounded cohomology is eventual coconnectedness.

Associated to a commutative DG-ring $A$ is the derived category of DG-modules over $A$, denoted by $\cat{D}(A)$.
For any $M \in \cat{D}(A)$, and any $n \in \mathbb{Z}$,
we have that $\mrm{H}^n(M)$ is an $\mrm{H}^0(A)$-module.
Given $M \in \cat{D}(A)$, we define the the (cohomological) infimum and supremum of $M$ by
\[
\inf(M) = \inf\{n \in \mathbb{Z} \mid \mrm{H}^n(M) \ne 0\},
\quad
\sup(M) = \sup\{n \in \mathbb{Z} \mid \mrm{H}^n(M) \ne 0\}.
\]
When both of these are finite, 
we define the (cohomological) amplitude of $M$ to be 
\[
\amp(M) = \sup(M) - \inf(M).
\]
In that case, we say that $M$ has bounded cohomology, and write $\amp(M) < \infty$.
We denote by $\cat{D}^{\mrm{b}}(A)$ the full subcategory of $\cat{D}(A)$ consisting of DG-modules with bounded cohomology.
It is a full triangulated subcategory of $\cat{D}(A)$.

We say that $A$ is noetherian if the ring $\mrm{H}^0(A)$ is a noetherian ring,
and for all $i<0$, the $\mrm{H}^0(A)$-module $\mrm{H}^i(A)$ is finitely generated.
This definition is justified by the derived Bass-Papp theorem, see \cite[Theorem 6.6]{ShInj}.
If $A$ is noetherian, and $M \in \cat{D}(A)$,
we say that $M$ has finitely generated cohomology if for all $n \in \mathbb{Z}$,
the $\mrm{H}^0(A)$-module $\mrm{H}^n(M)$ is finitely generated.
The full subcategory of $\cat{D}(A)$ consisting of DG-modules with finitely generated cohomology is denoted by $\mrm{D}_{\mrm{f}}(A)$. 
It is a triangulated subcategory of $\cat{D}(A)$.
The intersection $\mrm{D}_{\mrm{f}}(A) \cap \mrm{D}^{\mrm{b}}(A)$ is a triangulated subcategory denoted by $\mrm{D}^{\mrm{b}}_{\mrm{f}}(A)$.

A commutative noetherian DG-ring $A$ is called local if the commutative noetherian ring $\mrm{H}^0(A)$ is a noetherian local ring.
In that case, if $\bar{\m}$ is the maximal ideal of $\mrm{H}^0(A)$, 
we will say that $(A,\bar{\m})$ is a noetherian local DG-ring.
Given a commutative DG-ring $A$, 
and given $\bar{\p} \in \opn{Spec}(\mrm{H}^0(A))$,
denoting by $\pi:A^0 \to \mrm{H}^0(A)$ the natural surjection,
and letting $\p := \pi^{-1}(\bar{\p})$,
one sets $A_{\bar{\p}} := A \otimes_{A^0} A^0_{\p}$
This is a commutative DG-ring, called the localization of $A$ at $\bar{\p}$.
Similarly, if $\bar{S}$ is any multiplicatively closed set in $\mrm{H}^0(A)$,
letting $S = \pi^{-1}(\bar{S})$, 
we let $\bar{S}^{-1}A := A \otimes_{A^0} S^{-1}A^0$.
There is a natural localization functor $\cat{D}(A) \to \cat{D}(A_{\bar{\p}})$
given by $-\otimes_{A^0} A^0_{\p} \cong -\otimes^{\mrm{L}}_A A_{\bar{\p}}$. 
It is a triangulated functor.
Given $M \in \cat{D}(A)$, we denote its image in $\cat{D}(A_{\bar{\p}})$ by $M_{\bar{\p}}$.
The localization functor has the property that for any $n$ there is a natural isomorphism
$\mrm{H}^n(M_{\bar{\p}}) \cong \mrm{H}^n(M)_{\bar{\p}}$,
where the right hand side is the usual localization of the $\mrm{H}^0(A)$-module $\mrm{H}^n(M)$.
If $A$ is noetherian then $A_{\bar{\p}}$ is noetherian and local.

If $A$ is a commutative noetherian local DG-ring,
according to \cite[Theorem 0.1]{ShInj},
there is a full subcategory $\opn{Inj}(A) \subseteq \cat{D}(A)$,
which is the analogue of the category of injective modules over a ring.
In particular, the Matlis classification of injectives holds in this setting,
and (isomorphism classes) of elements of $\opn{Inj}(A)$ are in bijection with $\opn{Spec}(\mrm{H}^0(A))$.
We denote by $E(A,\bar{\p})$ the element of $\opn{Inj}(A)$ which corresponds to a given $\bar{\p} \in \opn{Spec}(\mrm{H}^0(A))$.
It is determined up to isomorphism by the properties that $\inf(E(A,\bar{\p})) > -\infty$, 
and 
\[
\mrm{R}\opn{Hom}_A(\mrm{H}^0(A),E(A,\bar{\p})) \cong E(\mrm{H}^0(A),\bar{\p}) \cong \mrm{H}^0\left(E(A,\bar{\p})\right),
\]
where $E(\mrm{H}^0(A),\bar{\p})$ is the usual injective hull of the $\mrm{H}^0(A)$-module $\mrm{H}^0(A)/\bar{\p}$.
See \cite{ShInj} for more details about the category $\opn{Inj}(A)$.

Given a commutative noetherian DG-ring $A$,
a DG-module $R \in \cat{D}_{\mrm{f}}(A)$ is called a dualizing DG-module
if it has finite injective dimension (in the sense of \cite[Definition 12.4.8]{YeBook}),
and the natural map $A \to \mrm{R}\opn{Hom}_A(R,R)$ is an isomorphism in $\cat{D}(A)$.
By \cite[Corollary 7.3]{Ye1}, 
if $\amp(A) < \infty$ then $\amp(R) < \infty$.
For any $M \in \cat{D}^{\mrm{b}}_{\mrm{f}}(A)$,
the natural map
\[
M \to \mrm{R}\opn{Hom}_A(\mrm{R}\opn{Hom}_A(M,R),R)
\]
is an isomorphism in $\cat{D}(A)$.
Given a commutative noetherian DG-ring $A$,
it follows from the proof of \cite[Theorem 7.9]{Ye1} that if there exist a Gorenstein noetherian ring $\K$ of finite Krull dimension, and a map of DG-rings $\K \to A$,
such that the induced map $\K \to \mrm{H}^0(A)$ is essentially of finite type,
then $A$ has a dualizing DG-module.
It follows that virtually all DG-rings that naturally arise in derived algebraic geometry have dualizing DG-modules.

A commutative noetherian local DG-ring $A$ is called Gorenstein if $A$ is a dualizing DG-module over itself. 
This implies that $A$ has bounded cohomology. 
A not necessarily local commutative noetherian DG-ring $A$ is called Gorenstein if for any $\bar{\p} \in \opn{Spec}(\mrm{H}^0(A))$,
the local DG-ring $A_{\bar{\p}}$ is Gorenstein.
See \cite{FIJ,FJ} for details about Gorenstein DG-rings.

According to \cite[Theorem 4.1]{ShCM},
if $A$ is a commutative noetherian DG-ring, 
and $R$ is a dualizing DG-module over $A$,
then $\amp(A) \le \amp(R)$.
A commutative noetherian local DG-ring $(A,\bar{\m})$ with bounded cohomology 
which has a dualizing DG-module $R$ is called local-Cohen-Macaulay if $\amp(A) = \amp(R)$.
In general, the local-Cohen-Macaulay property is defined even for local DG-rings which do not have dualizing DG-modules (either using local cohomology, or using regular sequences), but for local DG-rings which have dualizing DG-modules, the above definition is equivalent, and independent of the chosen dualizing DG-module.
See \cite[Theorem 2]{ShCM} for details.
A commutative noetherian DG-ring $A$ is called Cohen-Macaulay if for any $\bar{\p} \in \opn{Spec}(\mrm{H}^0(A))$,
the local DG-ring $A_{\bar{\p}}$ is local-Cohen-Macaulay.
Unlike the case of rings, for DG-rings it could happen that $A$ is local-Cohen-Macaulay but not Cohen-Macaulay.
Any Gorenstein DG-ring is Cohen-Macaulay.

Given a commutative noetherian local DG-ring $(A,\bar{\m})$,
and given $0 \ncong M \in \cat{D}^{\mrm{b}}_{\mrm{f}}(A)$,
an element $\bar{x} \in \bar{\m}$ is called $M$-regular
if it is $\mrm{H}^i(M)$-regular, where $i = \inf(M)$.
In other words, if the multiplication map 
\[
-\cdot \bar{x}: \mrm{H}^i(M) \to \mrm{H}^i(M)
\] 
is injective.

If $(A,\bar{\m})$ is a commutative noetherian local DG-ring,
and if $\bar{x} \in \bar{\m}$,
the identification $\mrm{H}^0(A) = \opn{Hom}_{\cat{D}(A)}(A,A) = \mrm{H}^0\left(\mrm{R}\opn{Hom}_A(A,A)\right)$
allows us to treat $\bar{x}$ as a morphism $A \to A$ in $\cat{D}(A)$.
The cone of this morphism is denoted by $A\sslash\bar{x}$.
This Koszul-type construction gives us quotients in the DG-setting.
It follows from \cite[Section 3.2]{Mi} that $A\sslash\bar{x}$ has the structure of a commutative noetherian local DG-ring,
with the property that $\mrm{H}^0(A\sslash\bar{x}) \cong \mrm{H}^0(A)/\bar{x}$.
If $M \in \cat{D}(A)$, we set $M\sslash\bar{x}M :=  A\sslash\bar{x} \otimes^{\mrm{L}}_A M$.
For a sequence $\bar{x}_1,\dots,\bar{x}_n \in \bar{\m}$, 
we define inductively 
\[
A\sslash(\bar{x}_1,\dots,\bar{x}_n) := (A\sslash\bar{x}_1)\sslash(\bar{x}_2,\dots,\bar{x}_{n}),
\]
by identifying $\bar{x}_2,\dots,\bar{x}_{n}$ with their images in $\mrm{H}^0(A\sslash\bar{x}_1)$.
We say that $\bar{x}_1,\dots\bar{x}_n \in \bar{\m}$ is an $A$-regular sequence if $\bar{x}_1$ is $A$-regular,
and $\bar{x}_2,\dots\bar{x}_n$ is $A\sslash\bar{x}_1$-regular.
It holds that $(A,\bar{\m})$ is local-Cohen-Macaulay if and only if there exist an $A$-regular sequence in $\bar{\m}$ of length $\dim(\mrm{H}^0(A))$.

We finish this section with the analogue of regular local rings in the setting of commutative DG-rings with bounded cohomology.
J{\o}rgensen showed in \cite[Theorem 0.2]{Jo} that if $A$ is a commutative noetherian local DG-ring with bounded cohomology,
and if $M \in \cat{D}^{\mrm{b}}_{\mrm{f}}(A)$ has finite projective dimension, 
then $\amp(M) \ge \amp(A)$. 
This implies that if we want all elements of $\cat{D}^{\mrm{b}}_{\mrm{f}}(A)$ to be of finite projective dimension,
then we must assume that $\amp(A) = 0$. 
In other words, in this setting, the only regular DG-rings are those which are quasi-isomorphic to regular rings.

\section{The regular and the Gorenstein loci}

We begin by discussing the Gorenstein locus of a commutative noetherian DG-ring.

\begin{thm}
Let $A$ be a commutative noetherian DG-ring. 
Assume that $\amp(A) < \infty$,
and that $A$ has a dualizing DG-module.
Then the set
\[
\opn{Gor}(A) = \{ \bar{\p} \in \opn{Spec}(\mrm{H}^0(A)) \mid A_{\bar{\p}} \mbox{ is a Gorenstein DG-ring}\}
\]
is an open subset of $\opn{Spec}(\mrm{H}^0(A))$.
\end{thm}
\begin{proof}
Let $\bar{\p} \in \opn{Gor}(A)$, so that $A_{\bar{\p}}$ is a Gorenstein DG-ring.
Let $R$ be a dualizing DG-module over $A$.
By \cite[Corollary 6.11]{Sh}, the localization $R_{\bar{\p}}$ is a dualizing DG-module over $A_{\bar{\p}}$.
According to \cite[Corollary 7.16]{Ye1},
the uniqueness of dualizing DG-modules over local DG-rings implies that there exists some $n \in \mathbb{Z}$ such that
$R_{\bar{\p}}[n] \cong A_{\bar{\p}}$. 
Let us replace $R$ by $R[n]$,
so that $R_{\bar{\p}} \cong A_{\bar{\p}}$.
Denote by $\alpha$ an isomorphism $\alpha:A_{\bar{\p}} \to R_{\bar{\p}}$ in $\cat{D}(A_{\bar{\p}})$,
and consider the isomorphisms
\begin{eqnarray}
\opn{Hom}_{\cat{D}(A)}(A,R) = \mrm{H}^0\left(\mrm{R}\opn{Hom}_A(A,R)\right) \cong \mrm{H}^0(R),\nonumber\\
\opn{Hom}_{\cat{D}(A_{\bar{\p}})}(A_{\bar{\p}},R_{\bar{\p}}) 
= \mrm{H}^0\left(\mrm{R}\opn{Hom}_{A_{\bar{\p}}}(A_{\bar{\p}},R_{\bar{\p}})\right) \cong \mrm{H}^0(R)_{\bar{\p}}.\nonumber
\end{eqnarray}

Using these isomorphisms, 
consider the following commutative diagram of $\mrm{H}^0(A)$-modules:
\[
\xymatrix
{
\opn{Hom}_{\cat{D}(A)}(A,R) \ar[r]\ar[d]_{\Phi} & \mrm{H}^0(R)\ar[d]^{\Psi}\\
\opn{Hom}_{\cat{D}(A_{\bar{\p}})}(A_{\bar{\p}},R_{\bar{\p}}) \ar[r] & \mrm{H}^0(R)_{\bar{\p}}
}
\]
Here, the horizontal maps are the above isomorphisms, 
the right vertical map $\Psi$ is the localization map,
and then the left vertical map $\Phi$ is the unique map making this diagram commutative.
It can be realized by the functor $-\otimes^{\mrm{L}}_A A_{\bar{\p}}$. 
The isomorphism $\alpha \in \opn{Hom}_{\cat{D}(A_{\bar{\p}})}(A_{\bar{\p}},R_{\bar{\p}})$ need not be in the image of $\Phi$.
However, as can be seen from the behavior of $\Psi$,
there exists $\beta \in \opn{Hom}_{\cat{D}(A)}(A,R)$ and $\bar{f} \in \mrm{H}^0(A)$,
such that $\bar{f} \notin \bar{\p}$, 
and such that $\Phi(\beta) = \bar{f} \cdot \alpha$.
Since $\bar{f} \notin \bar{\p}$, 
it is invertible in $\mrm{H}^0(A_{\bar{\p}})$,
so that the morphism $\alpha' := \bar{f} \cdot \alpha$ is also an isomorphism.
Let us complete $\beta$ to a distinguished triangle
\[
A \xrightarrow{\beta} R \to K \to A[1]
\]
in $\cat{D}(A)$. 
Since $A,R \in \cat{D}^{\mrm{b}}_{\mrm{f}}(A)$,
and it is a triangulated category, 
we deduce that the cone $K$ also belongs to $\cat{D}^{\mrm{b}}_{\mrm{f}}(A)$.
Applying the triangulated functor $-\otimes^{\mrm{L}}_A A_{\bar{\p}}$ to the above distinguished triangle,
we obtain a distinguished triangle
\[
A_{\bar{\p}} \xrightarrow{\alpha'} R_{\bar{\p}} \to K_{\bar{\p}} \to A_{\bar{\p}}[1]
\]
in $\cat{D}(A_{\bar{\p}})$. 
Since $\alpha'$ is an isomorphism, we deduce that $K_{\bar{\p}} \cong 0$.
Let
\[
V = \opn{Supp}\left(\mrm{H}^{*}(K)\right) := \bigcup_{i=\inf(K)}^{\sup(K)} \opn{Supp}\left(\mrm{H}^i(K)\right).
\]
The fact that $K$ is a bounded DG-module implies that $-\infty < \inf(K) \le \sup(K) < +\infty$,
so that $V$ is a closed subset of $\opn{Spec}(\mrm{H}^0(A))$.
Since $K_{\bar{\p}} \cong 0$, we deduce that $\bar{\p} \notin V$.
For any $\bar{\q} \in \opn{Spec}(\mrm{H}^0(A))$ such that $\bar{\q} \notin V$,
we have that $K_{\bar{\q}} \cong 0$, so that $A_{\bar{\q}} \cong R_{\bar{\q}}$.
By \cite[Corollary 6.11]{Sh}, $R_{\bar{\q}}$ is a dualizing DG-module over $A_{\bar{\q}}$,
so we deduce that $A_{\bar{\q}}$ is a Gorenstein DG-ring.
Let $W$ denote the complement of $V$ in $\opn{Spec}(\mrm{H}^0(A))$.
Then $W$ is an open set, $\bar{\p} \in W$, and $W \subseteq \opn{Gor}(A)$,
so that $\opn{Gor}(A)$ is an open set.
\end{proof}

Next, we would like to show that the open set $\opn{Gor}(A)$ is often empty.
To do that, we first need some auxiliary results.

\begin{prop}\label{prop:GorRegSeq}
Let $(A,\bar{\m})$ be a commutative noetherian local DG-ring with bounded cohomology,
and let $\bar{x} \in \bar{\m}$ be an $A$-regular element.
Then $A$ is Gorenstein if and only if $A\sslash\bar{x}$ is Gorenstein.
\end{prop}
\begin{proof}
Since $A$ is Gorenstein if and only if $\injdim_A(A) < \infty$,
the result follows immediately from the inequalities
\[
\injdim_A(A) - 1 \le \injdim_{A\sslash\bar{x}}(A\sslash\bar{x}) \le \injdim_A(A)
\]
proved in \cite[Corollary 3.32]{Mi}.
\end{proof}

The next result is a criterion for a local noetherian DG-ring $A$ with $\dim(\mrm{H}^0(A)) = 0$ to be Gorenstein.
This generalizes \cite[Theorem 3.1, $(i)^{\sim} \iff (iv)$]{AF},
where Avramov and Foxby proved a similar result for noetherian local DG-rings $A$ such that $\mrm{H}^0(A)$ is a field. 

\begin{prop}\label{prop:ZDGor}
Let $(A,\bar{\m})$ be a commutative noetherian local DG-ring,
and suppose that $\dim(\mrm{H}^0(A)) = 0$ and $n = \amp(A) < \infty$.
Let $E = E(A,\bar{\m})$.
Then $A$ is a Gorenstein DG-ring if and only if
there is an isomorphism $A \cong E[n]$ in $\cat{D}(A)$.
In particular, if $A$ is Gorenstein,
then for all $0 \le i \le n$, 
there is an isomorphism of $\mrm{H}^0(A)$-modules:
\[
\mrm{H}^{-i}(A) \cong \opn{Hom}_{\mrm{H}^0(A)}(\mrm{H}^{-n+i}(A),\mrm{H}^0(E)).
\]
\end{prop}
\begin{proof}
Since the local ring $\mrm{H}^0(A)$ is artinian, 
it follows that $E(\mrm{H}^0(A),\bar{\m})$, 
which is just the injective hull of $\mrm{H}^0(A)/\bar{\m}$ is a dualizing complex over $\mrm{H}^0(A)$.
It follows from the definition of $E(A,\bar{\m})$ that
\[
\mrm{R}\opn{Hom}_A(\mrm{H}^0(A),E(A,\bar{\m})) \cong E(\mrm{H}^0(A),\bar{\m}),
\]
and the fact that $\mrm{H}^0(A)$ is artinian implies that $E(A,\bar{\m}) \in \cat{D}^{+}_{\mrm{f}}(A)$.
Hence, by \cite[Proposition 7.20]{ShInj}, the DG-module $E(A,\bar{\m})$ is a dualizing DG-module over $A$.
Since, according to \cite[Corollary 7.16]{Ye1}, dualizing DG-modules over local DG-rings are unique up to a shift,
we deduce that $A$ is Gorenstein, 
that is, that $A$ is a dualizing DG-module over itself if and only if $A \cong E[n]$ in $\cat{D}(A)$.
Assuming now this is the case, 
for all $0 \le i \le n$, we have that
\[
\mrm{H}^{-i}(A) \cong \mrm{H}^{-i}(E[n]) = \mrm{H}^{n-i}(E) \cong \opn{Hom}_{\mrm{H}^0(A)}(\mrm{H}^{-n+i}(A),\mrm{H}^0(E)),
\]
where the last isomorphism follows from \cite[Corollary 4.12]{ShInj}.
\end{proof}

Combining these two propositions, we obtain the following criterion for being Gorenstein:
\begin{cor}\label{Cor:Gor}
Let $(A,\bar{\m})$ be a commutative noetherian local DG-ring such that $n = \amp(A) < \infty$.
Then $A$ is Gorenstein if and only if $A$ is a Cohen-Macaulay DG-ring,
and for some maximal $A$-regular sequence $\bar{x}_1,\dots,\bar{x}_d \in \bar{\m}$,
there is an isomorphism
\[
A\sslash(\bar{x}_1,\dots,\bar{x}_d) \cong E(A\sslash(\bar{x}_1,\dots,\bar{x}_d),\bar{\m}/(\bar{x}_1,\dots,\bar{x}_d))[n]
\]
in $\cat{D}(A\sslash(\bar{x}_1,\dots,\bar{x}_d))$.
\end{cor}
\begin{proof}
Assuming $A$ is Cohen-Macaulay, by \cite[Corollary 5.21]{ShCM}, 
there is a maximal $A$-regular sequence $\bar{x}_1,\dots,\bar{x}_d \in \bar{\m}$ of length $d = \dim(\mrm{H}^0(A))$,
and moreover the sequence $\bar{x}_1,\dots,\bar{x}_d \in \bar{\m}$ is a system of parameters of $\mrm{H}^0(A)$.
Then, by Proposition \ref{prop:GorRegSeq},
the DG-ring $A\sslash(\bar{x}_1,\dots,\bar{x}_d)$ is Gorenstein if and only if $A$ is Gorenstein,
and the latter has 
\[
\dim(\mrm{H}^0(A\sslash(\bar{x}_1,\dots,\bar{x}_d))) = 0.
\]
The result now follows from Proposition \ref{prop:ZDGor}.
\end{proof}

Using the above characterization of Gorenstein DG-rings,
we now show that it is often the case that commutative noetherian DG-rings have empty Gorenstein loci:

\begin{prop}\label{prop:GorEmpty}
Let $B$ be a commutative noetherian ring.
Then there exists a commutative noetherian DG-ring $A$ with $\mrm{H}^0(A) = B$,
such that $\amp(A) < \infty$, and such that 
\[
\opn{Gor}(A) = \{ \bar{\p} \in \opn{Spec}(\mrm{H}^0(A)) \mid A_{\bar{\p}} \mbox{ is a Gorenstein DG-ring}\} = \emptyset.
\]
\end{prop}
\begin{proof}
For any commutative noetherian ring $B$,
let $T(B)$ be the graded ring with 
\[
T(B)^0 = B, \quad T(B)^{-2} = B\oplus B, \quad T(B)^i = 0 \mbox{ for all $i \ne 0,-2$.}
\]
This is a graded-commutative ring in an obvious way, 
and we make it a commutative DG-ring by letting the differential be zero.
Observe that $\mrm{H}^0(T(B)) = B$,
that for any $\p \in \opn{Spec}(B)$, 
we have that $T(B)_{\p} \cong T(B_{\p})$,
and that if $(B,\m)$ is local and $x \in \m$ is $T(B)$-regular,
then $T(B)\sslash x \cong T(B/x)$.
Let $A = T(B)$. 
We claim that  $\opn{Gor}(A) = \emptyset$.
One way to show this is using \cite[Theorem 2.2]{J}.
Here we offer for this particular case an alternative proof. 
Suppose this is not the case, and let $\p \in \opn{Gor}(A)$.
Then $A_{\p} \cong T(B_{\p})$ is Gorenstein.
By Corollary \ref{Cor:Gor}, 
there are $x_1,\dots,x_d \in \p\cdot B_{\p}$,
such that the sequence $x_1,\dots,x_d$ is $A_{\p}$-regular,
and the DG-ring $C = A_{\p}\sslash(x_1,\dots,x_d)$ is Gorenstein with $\dim(\mrm{H}^0(C)) = 0$.
Letting $D = B_{\p}/(x_1,\dots,x_d)$,
we have that $C \cong T(D)$.
Since $C$ is Gorenstein, 
letting $E$ denote the injective hull of the residue field of $D$,
it follows from Proposition \ref{prop:ZDGor} that
\[
D \oplus D \cong H^{-2}(C) \cong \opn{Hom}_D(\mrm{H}^0(C),E) \cong \opn{Hom}_D(D,E) \cong E.
\]
So that
\[
\opn{Hom}_D(E,E) \cong \opn{Hom}_D(D\oplus D,D\oplus D) \cong D \oplus D \oplus D \oplus D.
\]
However, since $D$ is an artinian local ring,
by Matlis duality we have that $\opn{Hom}_D(E,E) \cong D$.
This is a contradiction, so $\opn{Gor}(A) = \emptyset$.
\end{proof}

\begin{rem}
Since there are obvious maps of DG-rings $B \to A \to B$ whose composition is $1_B$,
it follows from \cite[Proposition 7.5]{Ye1} that $A$ has a dualizing DG-module if and only if $B$ has a dualizing complex.
\end{rem}

We finish this section with a result about the regular locus of a commutative noetherian DG-ring with bounded cohomology.
Recall from our discussion in the preliminaries section 
that such a DG-ring is regular if and only if it is quasi-isomorphic to a regular ring.
This suggests the following definition for the regular locus of a commutative noetherian DG-ring $A$ with bounded cohomology:
\[
\opn{Reg}(A) = \{\bar{\p} \in \opn{Spec}(\mrm{H}^0(A)) \mid A_{\bar{\p}} \mbox{ is quasi-isomorphic to a regular local ring}\}
\]

A commutative noetherian ring is called J-1 if its regular locus is open.
See \cite[Section 32.B]{Mat} for a discussion about the J-conditions for rings,
and \cite{IT} for recent results about these properties.

\begin{prop}
Let $A$ be a commutative noetherian DG-ring. 
Assume that $\amp(A) < \infty$,
and that the noetherian ring $\mrm{H}^0(A)$ is J-1.
That is, assume that the set
\[
\opn{Reg}(\mrm{H}^0(A)) = \{\bar{\p} \in \opn{Spec}(\mrm{H}^0(A)) \mid {H}^0(A)_{\bar{\p}} \mbox{ is a regular local ring}\}
\]
is an open subset of $\opn{Spec}(\mrm{H}^0(A))$.
Then the set
\[
\opn{Reg}(A) = \{\bar{\p} \in \opn{Spec}(\mrm{H}^0(A)) \mid A_{\bar{\p}} \mbox{ is quasi-isomorphic to a regular local ring}\}
\]
is an open subset of $\opn{Spec}(\mrm{H}^0(A))$.
\end{prop}
\begin{proof}
Assume $\amp(A) = n$,
and for each $0 < i \le n$,
let $V_i = \opn{Supp}(\mrm{H}^{-i}(A))$.
Let $V = \bigcup_{i=1}^n V_i$.
Since $n < \infty$, 
and since each $\mrm{H}^{-i}(A)$ is finitely generated over $\mrm{H}^0(A)$,
it follows that $V$ is a closed subset of $\opn{Spec}(\mrm{H}^0(A))$.
Let $W$ be the complement of $V$. 
Then $W$ is a an open set,
and for any $\bar{\p} \in W$ and any $i<0$,
we have that 
\[
\mrm{H}^i(A_{\bar{\p}}) \cong \mrm{H}^i(A)_{\bar{\p}} = 0.
\]
It follows that for any $\bar{\p} \in W$,
the natural map $A_{\bar{\p}} \to \mrm{H}^0(A_{\bar{\p}})$ is a quasi-isomorphism,
while for any $\bar{\p} \notin W$,
we have that $\amp(A_{\bar{\p}}) > 0$.
From this we may deduce that
\[
\opn{Reg}(A) = W \cap \opn{Reg}(\mrm{H}^0(A)),
\]
so being the intersection of two open sets, 
we deduce that $\opn{Reg}(A)$ is open.
\end{proof}

\begin{rem}
Since any regular local ring is Gorenstein,
it follows that if $A$ is a commutative noetherian DG-ring 
such that $\amp(A) < \infty$,
then $\opn{Reg}(A) \subseteq \opn{Gor}(A)$.
Hence, Proposition \ref{prop:GorEmpty} implies that both of these sets could be empty,
no matter how nice $\mrm{H}^0(A)$ is.
\end{rem}

\section{The Cohen-Macaulay locus}

We now wish to study the Cohen-Macaulay locus of a 
commutative noetherian DG-ring with bounded cohomology which has a dualizing DG-module.
To do this, we will need to give a global criteria for such a DG-ring to be Cohen-Macaulay.
Before we do that, let us recall the situation for noetherian rings.
The next result is well known so we omit its proof.

\begin{prop}\label{prop:cmGlobal}
Let $A$ be a commutative noetherian ring,
and let $R$ be a dualizing complex over $A$.
\begin{enumerate}
\item If $\amp(R) = 0$ then $A$ is a Cohen-Macaulay ring.
\item If $\opn{Spec}(A)$ is connected and $A$ is a Cohen-Macaulay ring then $\amp(R) = 0$.
\end{enumerate}
\end{prop}

Recall that for a commutative ring $A$,
its spectrum $\opn{Spec}(A)$ is called irreducible if $A$ contains a unique minimal prime ideal.
Here is a DG-version of Proposition \ref{prop:cmGlobal}.

\begin{thm}\label{thm:cmGlobal}
Let $A$ be a commutative noetherian DG-ring. 
Suppose that $n = \amp(A) < \infty$,
and that $A$ has a dualizing DG-module $R$.
\begin{enumerate}
\item Assume that $\amp(A) = \amp(R)$ and that one of the following holds:
\begin{enumerate}
\item $\opn{Spec}(\mrm{H}^0(A))$ is irreducible.
\item $\opn{Supp}(\mrm{H}^{-n}(A)) = \opn{Spec}(\mrm{H}^0(A))$.
\end{enumerate}
Then $A$ is a Cohen-Macaulay DG-ring.
\item Assume that $\opn{Spec}(\mrm{H}^0(A))$ is equidimensional and $A$ is a Cohen-Macaulay DG-ring.
Then $\amp(A) = \amp(R)$.
\end{enumerate}
\end{thm}
\begin{proof}
\leavevmode
\begin{enumerate}[wide, labelwidth=!, labelindent=0pt]
\item  Let us first assume that $\opn{Spec}(\mrm{H}^0(A))$ is irreducible.
Let $\bar{\m}$ be some maximal ideal of $\opn{Spec}(\mrm{H}^0(A))$,
such that $\bar{\m} \in \opn{Supp}(\mrm{H}^{-n}(A))$.
Then 
\[
\amp(A_{\bar{\m}}) = \amp(A) = \amp(R) \ge \amp(R_{\bar{\m}}).
\]
Since by \cite[Theorem 4.1]{ShCM},
the opposite inequality $\amp(A_{\bar{\m}}) \le \amp(R_{\bar{\m}})$ always holds,
we deduce that $\amp(A_{\bar{\m}}) = \amp(R_{\bar{\m}})$,
so that $A_{\bar{\m}}$ is local-Cohen-Macaulay.
By \cite[Proposition 4.11]{ShCM},
this implies that
\[
\dim(\mrm{H}^{-n}(A_{\bar{\m}})) = \dim(\mrm{H}^0(A_{\bar{\m}})).
\]
Since $\opn{Spec}(\mrm{H}^0(A))$ is irreducible,
we have that
\[
\dim(\mrm{H}^0(A_{\bar{\m}})) = \dim(\mrm{H}^0(A)_{\bar{\m}}) = \dim(\mrm{H}^0(A)).
\]
We deduce that $\dim(\mrm{H}^{-n}(A)) = \dim(\mrm{H}^0(A))$.
The fact that $\opn{Spec}(\mrm{H}^0(A))$ is irreducible then implies that
$\opn{Supp}(\mrm{H}^{-n}(A)) = \opn{Spec}(\mrm{H}^0(A))$.
Under this assumption, as above,
for any $\bar{\p} \in \opn{Spec}(\mrm{H}^0(A))$,
we have that
\[
\amp(A_{\bar{\p}}) = \amp(A) = \amp(R) \ge \amp(R_{\bar{\p}}),
\]
so that that $\amp(A_{\bar{\p}}) = \amp(R_{\bar{\p}})$.
Hence, $A_{\bar{\p}}$ is local-Cohen-Macaulay,
so that $A$ is Cohen-Macaulay.
\item Set $a = \inf(R), b = \sup(R)$.
Let $S = \mrm{R}\opn{Hom}_A(\mrm{H}^0(A),R)$.
By \cite[Proposition 7.5]{Ye1}, $S$ is a dualizing complex over $\mrm{H}^0(A)$,
while by \cite[Proposition 3.3]{ShInj}, 
we have that $\inf(S) = a$, and $\mrm{H}^a(S) \cong \mrm{H}^a(R)$.
Since $\opn{Spec}(\mrm{H}^0(A))$ is equidimensional,
by \cite[Tag 0AWK]{SP}, we have that
\[
\opn{Supp}(\mrm{H}^a(S)) = \opn{Spec}(\mrm{H}^0(A)),
\]
so that
\[
\opn{Supp}(\mrm{H}^a(R)) = \opn{Spec}(\mrm{H}^0(A)).
\]
Take some $\bar{\p} \in \opn{Supp}(\mrm{H}^b(R))$.
Then $\amp(R_{\bar{\p}}) = \amp(R)$.
But $A_{\bar{\p}}$ is local-Cohen-Macaulay, 
so that
\[
\amp(A_{\bar{\p}}) = \amp(R_{\bar{\p}}).
\]
We deduce that
\[
\amp(A) \ge \amp(A_{\bar{\p}}) = \amp(R_{\bar{\p}}) = \amp(R).
\]
Since by \cite[Theorem 4.1]{ShCM} we have that $\amp(A) \le \amp(R)$,
we deduce that $\amp(A) = \amp(R)$.
\end{enumerate}
\end{proof}

Because the above result provides a global characterization for being 
Cohen-Macaulay only in the case where the spectrum is irreducible,
to use it to show that the Cohen-Macaulay property holds generically,
we first show that the property of being irreducible holds generically for noetherian rings.

\begin{prop}\label{prop:genIrred}
Let $A$ be a commutative noetherian ring.
There there exists a dense open set $W \subseteq \opn{Spec}(A)$,
such that $W$ has an open cover $W = \bigcup_{i=1}^m D(f_i)$,
with the property that for each $1 \le i \le m$,
the localization $A_{f_i}$ has an irreducible spectrum 
(that is, it has a minimal prime ideal).
In particular, for any $\p \in W$,
the ring $A_{\p}$ has an irreducible spectrum.
\end{prop}
\begin{proof}
Since $A$ is noetherian, it has only finitely many minimal prime ideals.
Let us denote the minimal primes by $\p_1,\dots,\p_n$.
For each $1 \le i \le n$,
let $\q_i$ be the intersection of all minimal primes except $\p_i$.
That is,
\[
\q_i = \p_1 \cap \dots \cap \p_{i-1} \cap \p_{i+1} \cap \dots \cap \p_n.
\]
Observe that minimality implies that $\q_i \nsubseteq \p_i$.
Let 
\[
V_i = V(\q_i) = \{\p \in \opn{Spec}(A) \mid \q_i \subseteq \p\},
\]
and let $W_i$ be the complement of $V_i$.
Then $W_i$ is an open set and $\p_i \in W_i$.
Since any prime ideal contains at least one minimal prime,
it follows that for $i \ne j$, $V_i \cup V_j = \opn{Spec}(A)$, 
so that $W_i \cap W_j = \emptyset$.
Let $W = \bigcup_{i=1}^n W_i$. 
This is an open subset of $\opn{Spec}(A)$, 
and since it contains all minimal prime ideals, it is dense.
For any $1 \le i \le n$, 
and for any $f \in A$,
if $D(f) \subseteq W_i$,
then the ring $A_f$ contains only one minimal prime ideal,
namely $\p_i$. 
Thus, the spectrum of $A_f$ is irreducible, 
so the result follows.
\end{proof}

Here is the main result of this section.

\begin{thm}
Let $A$ be a commutative noetherian DG-ring. 
Assume that $\amp(A) < \infty$,
and that $A$ has a dualizing DG-module.
Then $A$ is generically Cohen-Macaulay:
there exists a dense open set $W \subseteq \opn{Spec}(\mrm{H}^0(A))$,
such that for any prime ideal $\bar{\p} \in W$,
the DG-ring $A_{\bar{\p}}$ is Cohen-Macaulay.
Hence, the set
\[
\opn{CM}(A) = \{\bar{\p} \in \opn{Spec}(\mrm{H}^0(A)) \mid A_{\bar{\p}} \mbox{ is a Cohen-Macaulay DG-ring}\}
\]
contains a dense open set.
\end{thm} 
\begin{proof}
Letting $W = \bigcup_{i=1}^m D(f_i)$ be as in Proposition \ref{prop:genIrred},
it is enough to show that for each $1\le i \le m$,
the set $D(f_i) \cap \opn{CM}(A)$ contains a set which is open and dense in $D(f_i)$.
Replacing $A$ by the localization $(f_i)^{-1}A$, 
we may thus assume without loss of generality that $\opn{Spec}(\mrm{H}^0(A))$ is irreducible.
Let $a = \inf(R), b = \sup(R)$.
By the proof of Theorem \ref{thm:cmGlobal}(2),
we have that $\opn{Supp}(\mrm{H}^a(R)) = \opn{Spec}(\mrm{H}^0(A))$.
Let $\bar{\p}$ be the minimal prime ideal of $\opn{Spec}(\mrm{H}^0(A))$.
Since $\dim(\mrm{H}^0(A_{\bar{\p}})) = 0$,
by \cite[Proposition 4.8]{ShCM},
the DG-ring $A_{\bar{\p}}$ is local-Cohen-Macaulay,
so $\amp(A_{\bar{\p}}) = \amp(R_{\bar{\p}})$.
Note that $\inf(R_{\bar{\p}}) = a$,
and let $c = \sup(R_{\bar{\p}})$.
Then $c \le b$.
Let us set
\[
V = \bigcup_{i=c+1}^b \opn{Supp}(\mrm{H}^i(R)).
\]
This is a closed subset of $\opn{Spec}(\mrm{H}^0(A))$,
and $\bar{\p} \notin V$. 
If $U$ is the complement of $V$,
then $U$ is an open subset of $\opn{Spec}(\mrm{H}^0(A))$,
and since $\bar{\p} \in U$, 
we see that $U$ is dense.
Note that for any $\bar{\q} \in U$,
we have that
\[
\amp(R_{\bar{\q}}) \le \amp(R_{\bar{\p}}) = \amp(A_{\bar{\p}}).
\]
On the other hand, since $\bar{\p} \subseteq \bar{\q}$,
so that $A_{\bar{\p}}$ is a localization of $A_{\bar{\q}}$,
we deduce that $\amp(A_{\bar{\q}}) \ge \amp(A_{\bar{\p}})$.
Since by \cite[Theorem 4.1]{ShCM} we have that $\amp(A_{\bar{\q}}) \le \amp(R_{\bar{\q}})$,
it follows that $\amp(A_{\bar{\q}}) = \amp(R_{\bar{\q}})$.
Theorem \ref{thm:cmGlobal} now implies that $A_{\bar{\q}}$ is a Cohen-Macaulay DG-ring,
so that $\opn{CM}(A)$ contains a dense open set.
\end{proof}

\textbf{Acknowledgments.}

The author thanks Amnon Yekutieli for many useful conversations,
and for for encouraging me to prove Theorem \ref{thm:cmGlobal}.
The author is thankful to an anonymous referee for several suggestions that helped improving this manuscript.
This work has been supported by Charles University Research Centre program No.UNCE/SCI/022,
and by the grant GA~\v{C}R 20-02760Y from the Czech Science Foundation.

\end{document}